\documentclass[12pt,leqno]{article}
\usepackage{amssymb,amsmath,amsthm}
\usepackage[all]{xy}
\numberwithin{equation}{section}
\date{}

\usepackage{ran}

\begin{document}

\title{The Cauchy problem for $\shd$-modules on Ran spaces.}

\author{Giuseppe Bonavolont\`a}
\maketitle

\abstract{We will adopt an elementary approach to $\shd$-modules on Ran spaces in terms of two-limits; the aim here is to define the category of coherent $\shd$-modules, characteristic varieties and non-characteristic maps. An application will be the proof of the Cauchy-Kowaleski-Kashiwara theorem in this setting.}

\section{Introduction}
The pioneering work $\cite{BD04}$ on the geometric foundations of Quantum Field Theory and Conformal Field Theory stressed the importance played by so-called Ran spaces.
For example chiral, Lie* and factorization algebras can be considered as geometric objects living on such spaces (see $\cite{BD04}$). The most appropriate approach to the theory of $\shd$-modules on Ran spaces (as presented in $\cite{FG11}$) requires many sophisticated tools from Lurie's work on $(\infty,1)$-categories.
The main goal of this paper is to explore some definitions and properties of $\shd$-modules on Ran spaces employing the classical theory of prestacks and two limits ($\cite{SGA4}$). This simplified approach makes the theory accessible and sheds light on some definitions: the notions of characteristic varieties and of coherent $\shd$-modules are really natural in this language.
In paragraph $\S \ref{preliminaries}$ we recall some basics on the classical theory of $\shd$-modules and microlocal geometry; in $\S \ref{prestacks}$ we recall some definitions relative to prestacks, two-limits, etc.; in paragraph $\S \ref{paragraph Ran}$ we introduce the theory of modules on Ran manifolds by the use of $\S \ref{prestacks}$; as an application we prove the Cauchy-Kowaleski-Kashiwara theorem for Ran manifolds.

{\bf Acknowledgments.}
I am deeply grateful to Pierre Schapira for having called my attention on the study of coherent $\shd$-modules over Ran spaces.
I have much benefited from the reading of the unpublished manuscript \cite{GS06} and I wish to kindly thank the authors. Moreover I thank the Luxembourgian National Research Fund for support via AFR grant PhD 09-072.
\section{Preliminaries}\label{preliminaries}

\subsubsection*{Sheaves}
We shall mainly follow the notations of~\cite{KS90} for sheaves, that of \cite{KS06} for categories and that of \cite{Ka03} for $\shd$-modules.

Let $X$ be a real manifold and let  $\cor$ be a field.
One denotes by $\cor_X$ the constant sheaf on $X$ with stalk $\cor$
and  by $\md[\cor_X]$ the abelian category of sheaves  of $\cor$-modules on $X$.
We denote by $\Derb(\cor_X)$ the bounded derived category of $\md[\cor_X]$.
One simply calls an object of this category, ``a sheaf''.

We shall use the classical six operations on sheaves, the
internal hom $\hom$, the tensor product $\tens$, the direct image  $\oim{f}$, the proper direct image$\eim{f}$, the inverse image $\opb{f}$
and the derived functors $\rhom$, $\ltens$, $\roim{f}$, $\reim{f}$,  $\opb{f}$, as well as the extraordinary inverse image $\epb{f}$, right adjoint to $\reim{f}$.

We also use the notation $\etens$ for the external product: if $X$ and $Y$ are two manifolds and
where $q_1$ and $q_2$ denote the first and second projection on $X\times Y$, we set
for $F\in\md[\cor_X]$ and $G\in\md[\cor_Y]$, one sets
$F\etens G\eqdot \opb{q_1}F\tens\opb{q_2}G$.

 \subsubsection*{Microlocal geometry}
 For a real or complex manifold $X$, we denote by $\tau\cl TX\to X$ its tangent vector bundle and by
 $\pi\cl T^*X\to X$ its cotangent vector bundle. If $E\to X$ is a vector bundle, we identify $X$ with the zero-section.

 Let $f\cl X\to Y$ be  a morphism of  real or complex manifolds.
To $f$ are associated the tangent morphisms
\eq\label{diag:tgmor}
\xymatrix{
TX\ar[d]^-\tau\ar[r]^-{f'}&X\times_YTY\ar[d]^-\tau\ar[r]^-{f_\tau}&TY\ar[d]^-\tau\\
X\ar@{=}[r]&X\ar[r]^-f&Y.
}\eneq
By duality, we deduce the diagram:
\eq\label{diag:cotgmor}
&&\xymatrix{
T^*X\ar[d]^-\pi&X\times_Y\ar[d]^-\pi\ar[l]_-{f_d}\ar[r]^-{f_\pi}T^*Y
                                      & T^*Y\ar[d]^-\pi\\
X\ar@{=}[r]&X\ar[r]^-f&Y.
}\eneq
One sets
\eqn
&&T^*_XY\eqdot\ker f_d= \opb{f_d}(T^*_XX).
\eneqn

\begin{definition}
Let $\Lambda\subset T^*Y$ be a closed $\R^+$-conic subset. One says that
$f$ is non-characteristic for $\Lambda$, or else, $\Lambda$
is non-characteristic for $f$, if
\eqn
&&\opb{f_\pi}(\Lambda)\cap T^*_XY\subset X\times_YT^*_YY.
\eneqn
\end{definition}

\subsubsection*{$\sho$-modules}
Now assume that   $(X,\sho_X)$ is a complex manifold of complex dimension $d_{X}$. We denote by
$\md[\sho_X]$ the abelian category of sheaves of $\sho_{X}$-modules and by $\hom[\sho_X]$ and
$\tens[\sho_X]$ the internal hom and tensor product in this category.
The sheaf of rings  $\sho_{X}$ is Noetherian and we denote by $\mdcoh[\sho_X]$ the thick abelian subcategory of $\md[\sho_X]$ consisting of coherent sheaves.

We denote by $\Derb(\sho_X)$ the bounded derived category of  $\md[\sho_X]$ and by
$\Derb_\coh(\sho_X)$ the full triangulated subcategory consisting of objects with coherent cohomology.

Let $f\cl X\to Y$ be a morphism of complex manifolds.
We keep the notation $\oim{f}$ and $\eim{f}$ for the two direct image functors for $\sho$-modules and
we denote by $\spb{f}$ the inverse image functor for $\sho$-modules:
\eqn
&&\spb{f}\shg\eqdot \sho_X\tens[\opb{f}\sho_Y]\opb{f}\shg.
\eneqn
Hence we have the pair of adjoint functors $(f^{\ast},\oim{f})$
\eqn
\xymatrix{
\md[\sho_X]\ar@<+.7ex>^{\oim{f}}[r] & \md[\sho_{Y}]\ar@<+.7ex>^{\spb{f}}[l].
}
\eneqn
It follows that $f_{\ast}$ is left exact and $f^{\ast}$ is right exact. We denote by $\lspb{f}$ the left derived functor of
$\spb{f}$.

Given two manifolds $X$ and $Y$, we denote by $\detens$ the external product for $\sho$-modules. Hence,
\eqn
&&\shf\detens\shg\eqdot \sho_{X\times Y}\tens[\sho_X\etens\sho_Y](\shf\etens\shg).
\eneqn

 \subsubsection*{$\shd$-modules}
 We denote by $\shd_{X}$ the sheaf of rings of holomorphic differential operators, the subalgebra
 of ${\she nd}(\sho_{X})$ generated by $\sho_{X}$ and $\Theta_X$, the sheaf of holomorphic vector fields.

Unless otherwise specified, a $\shd_X$-module is
a left-$\shd_X$-module. Hence a right $\shd_X$-module is a $(\shd_X^\rop)$-module.
We denote by $\md[\shd_X]$ the abelian category of $\shd_X$-modules and by $\Derb(\shd_X)$ its bounded derived category.

Recall the operations
 \eqn
&&\hom[\sho_X]: \md[\shd_X]^\rop\times\md[\shd_X]\to\md[\shd_X],\\
&&\hom[\sho_X]: \md[\shd^\rop_X]^\rop\times\md[\shd^\rop_X]\to\md[\shd_X],\\
&&\tens[\sho_X]: \md[\shd_X]\times\md[\shd_X]\to\md[\shd_X],\\
&&\tens[\sho_X]: \md[\shd^\rop_X]\times\md[\shd_X]\to\md[\shd^\rop_X].
\eneqn
We denote by $\detens$ the external product for $\shd$-modules:
\eqn
&&\shm\detens\shn\eqdot \shd_{X\times Y}\tens[\shd_X\etens\shd_Y](\shm\etens\shn).
\eneqn
For a morphism $f:X\to Y$ of complex manifolds, we denote as usual
by $\shd_{X\rightarrow Y}$ the transfert bimodule, a $(\shd_X,f^{-1}\shd_{Y})$-bimodule
\eqn
&&\shd_{X\to Y}\eqdot\sho_X\otimes _{f^{-1}\sho_Y}f^{-1}\shd_{Y}.
\eneqn
One should be aware that the left $\shd_{X}$ structure of $\shd_{X\rightarrow Y}$ is not the one induced by
that of $\sho_X$.
We then have the inverse and direct image functors
\eqn
&& \dopb{f}:\Derb(\shd_Y)\rightarrow \Derb(\shd_X),\quad
\dopb{f}\shn=\shd_{X\to Y}\ltens[\opb{f}\shd_Y]\opb{f}\shn,\\
&&\doim{f}:\Derb(\shd^\rop_X)\rightarrow \Derb(\shd^\rop_Y),\quad\doim{f}(\shm)=\roim{f}(\shm\ltens[\shd_X]\shd_{X\to Y}).
\eneqn
We denote by $\mdcoh[\shd_X]$ the thick abelian subcategory of $\md[\shd_X]$ consisting of coherent $\shd_X$-modules  and by $\Derb_\coh(\shd_X)$ the full triangulated subcategory consisting of objects with coherent cohomology.

For $\shm\in \Derb_\coh(\shd_X)$, we denote by $\chv(\shm)$ its characteristic variety, a closed $\C^\times$-conic complex analytic subvariety of $T^*X$.

\subsubsection*{The Cauchy-Kowaleski-Kashiwara theorem}

\begin{definition}
Let $f:X\rightarrow Y$ be a morphisms of manifolds and $\shn$ be a coherent $\shd_{Y}$-module. One says that $f$ is  non-characteristic for $\shn$ if $f$ is non-characteristic for $\chv(\shn)$.
\end{definition}

\begin{proposition}\label{Ka1} {\rm (\cite{Ka70})}
Let $\shn$ be a coherent $\shd_{Y}$-module. Suppose that $f:X\rightarrow Y$ is
non-characteristic for $\shn$. Then
\banum
\item
$\operatorname{H}^{k}(\dopb{f}\shn)=0$ for $k\neq 0$,
\item
$\operatorname{H}^{0}(\dopb{f}\shn)$ is a coherent $\shd_{X}$-module,
\item
$ \mathrm{char}(\dopb{f}\shn)= f_{d}f_{\pi}^{-1}\mathrm{char}{\shn}.$
\eanum
\end{proposition}

For $\shn_1$ and $\shn_2$ in $\Derb(\shd_Y)$ we have a natural
morphism
\eqn
&&\opb{f}\rhom[\shd_{Y}](\shn_1,\shn_2)\to\rhom[\shd_{X}](\dopb{f}\shn_1, \dopb{f}\shn_2).
\eneqn
Note that
$\dopb{f}\sho_X\simeq \sho_{Y}$.

The following theorem is known as the {\it Cauchy-Kovalesky-Kashiwara} theorem:

\begin{theorem}\label{th:CKK} {\rm (\cite{Ka70})} Let $f\cl X\rightarrow Y$ be a
  morphism of complex manifolds and $\shn$ be a coherent $\shd_{Y}$-module.
If $f$ is non-characteristic for $\shn$, then
\beq
f^{-1}\rhom[\shd_{Y}](\shn,\sho_{Y})\rightarrow \rhom[\shd_{X}](\dopb{f}\shn,\sho_{X})
\eeq
is an isomorphism.
\end{theorem}

\section{Grothendieck stacks}\label{prestacks}
References for this section are made to \cite{SGA4} (see also~\cite{KS06} for an exposition).
The results of this section are well-known to the specialists.
We will work in a given universe $\shu$, a category means a
$\shu$-category
and a set means a small set. Here
$\Cat$ denotes the big $2$-category of all $\shu$-categories.
\subsubsection*{Prestack}
Let $\sts$ be a $2$-projective system of categories indexed by a small
category $\cI$, that is, a functor $\sts:\cI^\rop\to\Cat$. For short, we consider $\cI$ as a presite and call
$\sts$ a prestack on $\cI$. Hence
\banum
 \item for each $i\in\cI$, $\sts(i)$ is a
category;
\item for each morphism $s\cl i_1\to i_2$ in $\cI$, $\sts(s)$ is a functor
$\sts(i_2)\to\sts(i_1)$, called the restriction functor (for
convenience, we shall write in the sequel $\restr_s$ instead of $\sts(s)$);

\item for $s\cl i_1\to i_2$ and $t\cl i_2\to i_3$, we have an isomorphism of
functor $c_{s,t}\cl \rho_s\circ \rho_t \isoto \rho_{t\circ s}$, making the
diagram below commutative for each $u\cl i_3\to i_4$:
\eqn
&&\xymatrix{
\rho_s\circ\rho_t\circ\rho_u\ar[r]^-{c_{t,u}}\ar[d]_-{c_{s,t}}&\rho_s\circ\rho_{u\circ t}\ar[d]^{c_{s,u\circ t}}\\
\rho_{t\circ s}\circ\rho_u\ar[r]^{c_{t\circ s,u}}&\rho_{u\circ t\circ s}.
}\eneqn
\item finally $\rho_{\operatorname{id}_{i}}=\operatorname{id}_{\sts(i)}$ and $c_{\operatorname{id}_{i},\operatorname{id}_{i}}=\operatorname{id}_{\operatorname{id}_{\sts{(i)}}}$ for any $i\in \cI$.
\eanum

\subsubsection*{Functor of prestacks}
Let $\sts_{\nu}$ ($\nu=1,2$) be prestacks on $\cI$ with the restrictions $\restr^{\nu}_s$ and the composition isomorphisms $c^{\nu}_{s,t}$.\\
Recall that a functor of prestacks $\Phi\cl\sts_{1}\to \sts_{2}$ on $\cI$ is the data of:
\banum
\item
for any $i\in \cI$, a functor $\Phi(i)\cl \sts_{1}(i)\to \sts_{2}(i)$;
\item
for any morphism $s\cl i_1\to i_2$, an isomorphism $\Phi_{s}$ of functors from $\sts_{1}(i_2)$ to $\sts_{2}(i_1)$
$$\Phi_{s}\cl \Phi(i_1)\circ \restr^{1}_{s}\isoto \restr^{2}_{s}\circ \Phi(i_2);$$
\eanum
these data satisfying: for any sequence of morphisms  $i_1\to[s]i_2\to[t]i_3$ the following diagram commutes
\eq\label{morp-prestack}
\vcenter{
\xymatrix{
\Phi(i_1)\circ \restr^1_s \circ \restr^1_t\ar[rr]^-{\Phi_{s}}
                                  \ar[d]^-{c^{1}_{s,t}}
      &&\restr^{2}_s\circ \Phi(i_2)\circ\restr^{1}_{t}\ar[rr]^-{\Phi_{t}}&& \restr^{2}_s\circ \restr^{2}_{t}\circ \Phi(i_3)\ar[d]^-{c^{2}_{s,t}}&&\\
  \Phi(i_1)\circ \restr^{1}_{t\circ s}\ar[rrrr]^-{\Phi_{t\circ s}}
     && &&\restr^{2}_{t\circ s}\circ\Phi(i_3).
    }  }
\eeq
\subsubsection*{Morphism of functors of prestacks}
We will need the following definition
\begin{definition}
Let $\Phi_{\nu}\cl\sts_{1}\to \sts_{2}$ ($\nu=1,2$) be two functors of prestacks on $\cI$.
A morphism of functors of prestacks $\theta\cl\Phi_{1}\to\Phi_{2}$ is the data for any $i\in \cI$ of a morphism of functors $\theta(i)\cl\Phi_{1}(i)\to \Phi_{2}(i)$
such that for any morphism $s\cl i_1 \to i_2$ in $\cI$, the following diagram commutes
\eq
\xymatrix{
\Phi_{1}(i_1)\circ \restr^1_s \ar[d]^-{\Phi^1_s} \ar[rr]^-{\theta(i_1)} &&\Phi_{2}(i_1)\circ \restr^1_s  \ar[d]^-{\Phi^2_s}&&\\
  \restr_{s}^2\circ \Phi_{1}(i_2) \ar[rr]^-{\theta(i_2)}
     && \restr^{2}_{s}\circ\Phi_{2}(i_2).
    }
\eneq

\end{definition}
\subsubsection*{$2$-projective limits}
For a  prestack $\sts$ on $\cI$, one defines the category
$\sts(\cI)$ as  the $2$-projective limit of the functor
$\sts\cl \cI^\rop\to \Cat$:
\eqn
&&\sts(\cI)=\tprolim[i,\,s]\sts(i).
\eneqn
More explicitly  $\sts(\cI)$ is given as follows.
\begin{definition}\label{def:2tprolim}
\banum
\item
An object $F$ of $\sts(\cI)$
 is a family $\{(F_i, \varphi_s)\}_{i,s}$
($i\in\cI$, $s\in\Fl(\cI)$) where
\bnum
\item
for any $i\in\cI$, $F_i$ is an object of $\sts(i)$,
\item
for any morphism $s\colon i_1\to i_2$ in $\cI$,
$\varphi_s \cl\restr_s F_{i_2}\isoto F_{i_1}$
is an isomorphism such that
\begin{itemize}
\item
for all $i \in \cI$, $\varphi_{\id_i} = \id_{F_i}$,
\item
for any sequence $i_1\to[s]i_2\to[t]i_3$ of morphisms in $\cI$, the following diagram commutes
\eq\label{cond:cos_S(A)} && \vcenter{
\xymatrix{
\restr_s \restr_tF_{i_3}\ar[rr]_-\sim^-{\restr_s(\varphi_t)}
                                  \ar[d]^-\wr_-{c_{s,t}}
      &&\restr_sF_{i_2}\ar[d]^-\wr_-{\varphi_{s}}\\
  \restr_{t\circ s}F_{i_3}\ar[rr]_-\sim^-{\varphi_{t\circ s}}
      &&F_{i_1}.
    } }  \eneq
\end{itemize}
\enum
\item
A morphism
$f\cl\{(F_i,\varphi_s)\}_{i,s}\to\{(F'_i,\varphi'_s)\}_{i,s}$
in $\sts(\cI)$ is a family of morphisms
$\{f_i\cl F_i\to F'_i\}_{i\in \cI}$ such that
for any $s\cl i_1\to i_2$, the diagram below commutes:
\eq\label{diag:morsts} \vcenter{ \xymatrix{
\restr_s(F_{i_2})\ar[d]^\wr_-{\varphi_s}\ar[rr]^-{\restr_s(f_{i_2})}
             &&\restr_{s}(F'_{i_2}) \ar[d]_\wr^-{\varphi'_s}  \\
F_{i_1}\ar[rr]^-{f_{i_1}}
&&F'_{i_1} .  }
}
\eneq
\eanum
\end{definition}

\begin{remark}\label{rem:morinsts+}
Denote by $\Fl(\cI)$ the category whose objects are the morphisms
of $\cI$, a morphism $(s\cl i\to j)\to (s'\cl i'\to j')$ being
visualized by the commutative diagram
\eqn
\xymatrix@R=4ex{
i\ar[r]^-s\ar[d]_-t&j\\
i'\ar[r]^-{s'}&j'.\ar[u]_-{t'}
}\eneqn
For two objects $F=\{(F_i,\varphi_s)\}_{i,s}$
and $F'=\{(F'_i, \varphi'_s)\}_{i,s}$ in $\sts(\cI)$,  consider  the
functor $\Psi\cl \Fl(\cI)^\rop\to\Set$ defined as follows.
For $(s\cl i\to j)\in\Fl(\cI)$, set
\eqn
&&\Psi(s)= \Hom[\sts(i)](F_i,\restr_sF'_j).
\eneqn
For $(s'\cl i'\to j')\in\Fl(\cI)$ and for a morphism
$(t,t')\cl s\to s'$ in $\Fl(\cI)$,
define $\Psi(s')\to\Psi (s)$ as the composition
\eqn
\Hom[\sts(i')](F_{i'},\restr_{s'}F'_{j'})
&\to &\Hom[\sts(i')](\restr_tF_{i'},\restr_t\restr_{s'}F'_{j'})\\
&\to &\Hom[\sts(i)](F_i,\restr_sF'_j)
\eneqn
where the last map is associated with the morphisms
$F_i\to \restr_tF_{i'}$ and
$\restr_t\restr_{s'}F'_{j'}\to\restr_t\restr_{s'}\restr_{t'}F'_j
\simeq \restr_sF'_j)$.

Then
\eqn
&&\Hom[\sts(\cI)](F,F') \simeq\prolim[s\in\Fl(\cI)^\rop]\Psi(s).
\eneqn
\end{remark}

\begin{proposition}\label{pro:prestack1}
Let $\sts$ be a prestack on $\cI$.
Assume that,
for each object $i\in\cI$, the category $\sts(i)$
 admits inductive \lp resp.\ projective\rp\, limits indexed by a small
 category $A$ and that,
for each morphism $s\cl i\to j$ in $\cI$, the restriction functor
$\restr_s$ commutes with such limits.
Then the category $\sts(\cI)$
 admits inductive \lp resp.\ projective\rp\, limits indexed by
 $A$ and the restriction functors $\restr_i\cl \sts(\cI)\to\sts(i)$
 commute with such limits.
\end{proposition}
\subsubsection*{Some Lemmas}
\begin{lemma}\label{lemma1}Every functor of prestacks $\Psi\cl\sts_{1}\to \sts_{2}$ on $\cI$ induces a functor $\Psi_{\infty}\cl\sts_{1}{(\cI)}\to \sts_{2}{(\cI)}$ on the corresponding two-limit categories.
\end{lemma}
\begin{proof}To each object $\{(F_i, \varphi_s)\}_{i,s} \in \sts_{1}(\cI)$, $\Psi$ associates $\{(\Psi{(i)}F_i,\tilde{ \varphi_s})\}_{i,s} \in \sts_{2}(\cI)$,
with $\tilde{\varphi_s}$ defined as the composition $\restr^{2}_{s}\Psi{(j)}F_j\isoto \Psi{(i)}\restr^{1}_{s}F_j\isoto \Psi{(i)}F_i$ where the first isomorphism is the inverse of $\Psi_{s}$ and the second is $\Psi{(i)}(\varphi_{s})$.
Notice that for any sequence $i_1\to[s]i_2\to[t]i_3$ of morphisms in $\cI$, the following diagram commutes
\eq
\xymatrix{
\restr^{2}_s \restr^2_t\Psi{(i_{3})} F_{i_3}\ar[rr]_-\sim^-{\restr^{2}_s(\tilde{\varphi_t})}
                                  \ar[d]^-\wr_-{c^2_{s,t}}
      &&\restr^{2}_s\Psi{(i_{2})}F_{i_2}\ar[d]^-\wr_-{\tilde{\varphi_{s}}}\\
  \restr^2_{t\circ s}\Psi{(i_{3})} F_{i_3}\ar[rr]_-\sim^-{\tilde{\varphi}_{t\circ  s}}
      &&\Psi{(i_{1})}F_{i_1};
    }   \eneq
in fact it is enough to apply $\Psi(i_{1})$ to the corresponding diagram in def $(\ref{def:2tprolim})$ and then use property $(\ref{morp-prestack})$ of the morphism $\Psi$.\\
Finally $\Psi_{\infty}$ associates to each morphism $f\cl\{(F_i,\varphi_s)\}_{i,s}\to\{(F'_i,\varphi'_s)\}_{i,s}$, $\Psi_{\infty}f=\{\Psi(i)f_{i}\} \cl\sts_{1}(\cI)\to\sts_{2}(\cI)$ with
\eq \vcenter{ \xymatrix{
\restr_s(\Psi(i_{2})F_{i_2})\ar[d]^\wr_-{\tilde{\varphi_s}}\ar[rr]^-{\restr_s\Psi(i_{2})(f_{i_2})}
             &&\restr_{s}(\Psi(i_{2})F'_{i_2}) \ar[d]_\wr^-{\tilde{\varphi'_s}}  \\
\Psi(i_{1})F_{i_1}\ar[rr]^-{\Psi(i_{1})f_{i_1}}
&&\Psi(i_{1})F'_{i_1}  }
}
\eneq
(it is enough to apply the functor $\Psi(i_{1})$ to the diagram (\ref{diag:morsts}))

\end{proof}
The next lemma is obvious
\begin{lemma}\label{lemma2}Let $\Psi_1\cl\sts_{1}\to \sts_{2}$ and $\Psi_2\cl\sts_{2}\to \sts_{3}$ be two morphisms of prestacks on $\cI$. There exists an isomorphism of functors $\Psi_{1\;\infty}\circ\Phi_{2\;\infty}\isoto(\Psi_1\circ \Phi_2)_{\infty}\cl\sts_{1}{(\cI)}\to \sts_{3}{(\cI)}$.
\end{lemma}
Define
\eq \pi\cl\sts(\cI)\to \prod_{i\in\cI}\sts(i) \eneq
as the functor that associates to each $\{(F_i, \varphi_s)\}_{i,s}$, the collection $\{F_i\}_{i}$ and acts on each morphism $\{f_{i}\}_{i\in I}$ forgetting the property $(\ref{diag:morsts})$.
\begin{lemma}\label{lemma3}
The functor $\pi$ is conservative.
\end{lemma}
\proof{
Consider $\{f_{i}\}_{i\in I}$ a morphism in $\sts(\cI)$ between $\{F_i,\varphi_s\}_{i,s}$ and $\{F'_i,\varphi'_s\}_{i,s}$.
Assume $\{f_{i}\}_{i\in I}$ is an isomorphism in $\prod_{i\in\cI}\sts(i)$.  For each $f_{i}$ denote by $g_i$ the corresponding inverse morphism. Consider the diagram
\eq\vcenter{\xymatrix{
\restr_s(F'_{i_2})\ar[d]^\wr_-{\varphi'_s}\ar[rr]^-{\restr_s(g_{i_2})}
             &&\restr_{s}(F_{i_2}) \ar[d]_\wr^-{\varphi_s}\ar[rr]^-{\restr_s(f_{i_2})} &&\ar[d]^\wr_-{\varphi'_s} \restr_s(F'_{i_2})\\
F'_{i_1}\ar[rr]^-{g_{i_1}}
&&F_{i_1}\ar[rr]^-{f_{i_1}}&&F'_{i_{1}};}
}\eneq
the right square is commutative by hypothesis and the large rectangle is trivially commutative. Then the equality $f_{i_{1}}\circ \varphi_{s}\circ \restr_s(g_{i_2})=f_{i_{1}}\circ g_{i_{1}}\circ\varphi'_{s}$ implies $\varphi_{s}\circ \restr_s(g_{i_2})= g_{i_{1}}\circ\varphi'_{s}$ since each $f_{i_{1}}$ is an isomorphism.
}
\begin{lemma}\label{lemma4}Let $\Psi_{\nu}\cl\sts_{1}\to \sts_{2}$ $(\nu=1,2)$ be two functors of prestacks on $\cI$, and
 $\theta\cl\Psi_{1}\to\Psi_{2}$ a morphisms of functors.
\banum
\item $\theta$ defines a morphism of functors $\theta_{\infty}\cl\Psi_{1\;\infty}\to\Psi_{2\;\infty}$,
\item if $\theta$ is an isomorphism, then $\theta_{\infty}$ is an isomorphism as well.
\eanum
\end{lemma}
{\proof Obvious.}
\subsubsection*{Grothendieck prestacks}\label{section:Gr1}

Recall that a $\cor$-abelian prestack $\sts$ on a presite $\cI$ is a
prestack such that $\sts(i)$ is a $\cor$-abelian category for each
$i\in\cI$ and the restriction functors $\restr_s$ are exact. For a site
$\cI$, a $\cor$-abelian stack is a $\cor$-abelian
prestack which is a stack.
If there is no risk of confusion, we shall not mention
the field $\cor$.

\begin{definition}\label{def:Grprestack}{\rm(see~\cite{GS06})}
\bnum
\item
A Grothendieck prestack $\sts$ over $\cor$ on a
presite $\cI$ is a $\cor$-abelian
prestack such that the abelian category $\sts(i)$
 is a Grothendieck category for each object $i\in\cI$
and the restriction functor
$\restr_s$ commutes with small inductive
limits for each morphism $s\cl i\to j$. (Recall that $\restr_s$ is exact.)
\item
For a site $\cI$, a Grothendieck stack is a Grothendieck prestack which is a
stack.
\enum
\end{definition}

The next result  is a deep results of~\cite{SGA4} (see Expos{\'e} I, Theorem~9.22).

\begin{theorem}\label{th:Grstack1}
Let $\sts$ be a Grothendieck prestack on $\cI$. Then
$\sts(\cI)$  is a  Grothendieck category.
\end{theorem}
If $\shc$ is an abelian category, we denote as usual by $\Derb(\shc)$
(resp.\ $\Derpb(\shc)$) its bounded
(resp.\ bounded from below) derived category. Hence, for each $i\in\cI$, the
functor $\restr_i\cl\sts(\cI)\to\sts(i)$ extends
as a functor $\tilde{\restr_i}\cl\Derpb(\sts^*(\cI))\to\Derpb(\sts^*(i))$ and one
checks easily that these functors define a functor
\eq\label{eq:derrho}
&&\tilde{\restr}\cl\Derpb(\sts(\cI))\to 2\prolim[i,s]\Derpb(\sts(i)).
\eneq

The $2\prolim$ of triangulated categories is no more triangulated in general; the functor $\tilde{\restr}$ is neither full nor faithful. However it would be interesting to check if this functor is conservative.
\section{Ran manifolds}\label{paragraph Ran}
We shall specialize the general theory of 2-limits presented in the previous section.
Let $\cI$ be a small category.
 \begin{definition}
 \banum
 \item
 A  real (resp. complex)  Ran-manifold indexed by $\cI$  is a projective system
 $X= \{X_i\}_{i\in\cI}$ of real (resp. complex) manifolds indexed by $\cI$
such that for every morphism $s\in \Hom[\cI](j,i)$ the corresponding morphism $\De_s:X_{i}\rightarrow X_{j}$ is a closed embedding.
\item
Let $X= \{X_{i}\}_{i\in\cI}$ and  $Y=\{Y_{i}\}_{i\in\cI}$ be two Ran-manifolds.
A morphism $f\cl X\to Y$  is defined as a collection of morphisms $\{f_i\}_{i\in\cI}$
such that all the diagrams below commute
\eq \label{mmorphism}
&&\xymatrix@C=5pc@R=3pc{
X_i\ar[r]^{f_i}\ar[d]_-{\De_s}&\ar[d]^-{\De_s}Y_i\\
X_j\ar[r]_{f_j}&Y_j
}\eneq
\eanum
\end{definition}

%

 For each morphism $s\cl j \to i$, with the corresponding embedding
$\De_s: X_{i}\hookrightarrow X_{j}$, we consider the following diagram
\beq\label{eq:diagDes}
\xymatrix@C=5pc@R=5pc{
T^{\ast}X_{i}\ar@{->}_-{\pi}[d] & \ar@{->}_-{\De_{sd}}[l]\ar@{->}^-{\pi}[d]
               X_{i}\times_{X_{j}}T^{\ast}X_{j}\ar@{^(->}^-{\De_{s\pi}}[r]&\ar@{->}^-{\pi}[d]T^{\ast}X_{j}\\
X_{i}\ar@{=}[r]&X_{i}\ar@{^(->}^{\De_s}[r]&X_{j}\\}
\eeq
We denote by $T^{\ast}_{s}X_{j}$ the kernel of $\De_{sd}$, that is, $\opb{\Delta_{sd}}T^*_{X_i}X_i$.
\begin{definition}\label{def:noncar1}
\banum
\item
We define $T^{\ast}X$ as the collection of all the diagrams~\eqref{eq:diagDes}.
\item
A family
$\{ {\Lambda_{i}} \}_{i\in \cI}$ of closed $\R^{+}$-conic subsets $\Lambda_{i}\subset T^{\ast}X_{i}$ satisfying
\beq
\Lambda_{i}\subset{\De_{sd}}(\opb{\De_{s\pi}}(\Lambda_{j}))
\eeq
for all $s\cl j\to i$ is said to be a closed $\R^{+}$-conic subset $\Lambda \subset T^{\ast} X$.
\item A closed $\R^{+}$-conic subset $\Lambda$ of $T^{\ast} X$ is said to be
transversal to the identity, if for each $s\cl j\to i$, the map $\Delta_s$ is
non characteristic with respect to $\Lambda_j$, that is,
$\opb{\De_{sd}}{T^*_{X_{i}}}X_{i}\cap\opb{\Delta_{s\pi}}\Lambda_j\subset
X_i\times_{X_j}T^*_{X_j}X_j$.
\eanum
\end{definition}

Consider a morphism $f\cl X\to Y$ of Ran-manifolds. It gives rise to the
commutative diagrams associated to $s\cl j\to i$:
\beq\label{eq:diagDes2}
\xymatrix@C=5pc@R=5pc{
T^*X_i&X_i\times_{X_j}T^*X_j\ar[l]\ar@{^(->}[r]&T^*X_j\\
X_i\times_{Y_i}T^*Y_i\ar_-{f_{id}}[u]\ar[d]^-{f_{i\pi}}
            &X_i\times_{Y_j}T^*Y_j\ar[l]\ar@{^(->}[r]\ar[u]\ar[d]\ar[lu]_-{\Delta_{f_s,d}}\ar[rd]^-{\Delta_{f_s,\pi}}
                                                   &X_j\times_{Y_j}T^*Y_j\ar_-{f_{jd}}[u]\ar[d]^-{f_{j\pi}}\\
T^*Y_i&Y_i\times_{Y_j}T^*Y_j\ar[l]\ar@{^(->}[r]&T^*Y_j.
}
\eeq
Notice that in Diagram~\ref{eq:diagDes2},  we have defined the maps
\eqn
&&\Delta_{f_s,d}\cl X_i\times_{Y_j}T^*Y_j\to T^*X_i,\quad
\Delta_{f_s,\pi}\cl X_i\times_{Y_j}T^*Y_j\to T^*Y_j.
\eneqn
where $f_s\cl X_{i}\to Y_{j}$ is given by the sequence of morphisms of manifolds $\xymatrix{X_{i}\ar[r]^{f_i}& Y_{i}\ar[r]^{\De^{(Y)}_s}& Y_{j}}$.
\begin{definition}\label{def:noncar2}
Let us fix a closed $\R^{+}$-conic subset $\Lambda \subset T^{\ast} Y$.

We say that $\Lambda$ is transversal to $f$ if for each $s\cl j\to i$, the map $f_s\cl X_{i}\to Y_{j}$ is non-characteristic with respect to $\Lambda_j$, that is
$\opb{\Delta_{f_s d}}(T^*_{X_i}X_i)\cap\opb{\Delta_{f_s \pi}}(\Lambda_j)\subset
X_i\times_{Y_j}T^*_{Y_j}Y_j$.
\end{definition}
Note that to be transversal is stronger than to be non characteristic.
When $f\cl X\to Y=X$ is the identity, we recover the notion of
Definition~\ref{def:noncar1}~(c).
\begin{notation}{We will denote by $\De^{(Y)}_s,\;\De^{(X)}_s,\;\forall s\in \Hom(\cI),$ the closed embeddings of $Y$ and $X$ respectively. However when it is clear from the context we will omit the superscript.}
\end{notation}
\begin{lemma}\label{transversal lemma} Let $\Lambda$ be a closed $\R^+$-conic subset of $T^*Y$; moreover assume $\Lambda$ is transversal to $f$. Then for every $s \cl j\to i$
\banum
\item $\De^{(Y)}_s$ is non-characteristic for $\Lambda_j$ in a neighborhood of $f_i(X_i)$;
\item $\De^{(X)}_s$ is non-characteristic for $f_{jd}f_{j\pi}^{-1}\big(\Lambda_j\big)$.
\eanum
\end{lemma}
\begin{proof}
\banum
\item

By hypothesis the morphism $f_s$ is non-characteristic for $\Lambda_{j}$. The result follows from Lemma $4.10$ pag.$65$ in \cite{Ka03} applied to the following commutative diagram:
\eq
\xymatrix{T^*X_{i}&\ar[l]_-{f_{id}}X_{i}\times_{Y_{i}}TY_{i}\ar[d]_{f_{i\pi}}&\ar[l]_{\varphi}X_{i}\times_{Y_{j}}T^*Y_{j}\ar[d]^{\psi}\\
& T^*Y_{i}&\ar[l]_{\De_{sd}}Y_{i}\times_{Y_{j}}T^*Y_{j}\ar[d]^{\De_{s\pi}}\\
& & T^*Y_{j}.}
\eneq
\item Same argument applied to $f_s=f_j \circ\De^{(X)}_s$.
\eanum

\end{proof}
 \section{Sheaves on Ran-manifolds}
Let $X$ be a Ran-manifold and $\cI$ a small category. Consider $\Derb(\cor_{X_i})$ the bounded derived category of $\md[\cor_{X_i}]$.
Define the functor $\sts\cl\cI^{op}\to \Cat$ as follows:
\banum
\item to any $i\in \cI$ set $\Derb(\cor_{X_{i}});$
\item for any morphism $s\cl i_{1}\to i_{2}$ set $\restr_s:\Derb(\cor_{X_{i_1}})\to\Derb(\cor_{X_{i_2}})$, with $\restr_s:=\opb{\De_s};$
\item for $s\cl i_{1}\to i_{2},\;t\cl i_{2}\to i_{3}$ the isomorphism of functors $c_{s,t}$ is
$\opb{\De_t}\circ\opb{\De_s}\simeq\opb{\De_{t\circ s}}$.
\eanum

 \begin{definition}\label{def:shvonR}
 We set
 \eqn
\Derb(\cor_X)\eqdot\tprolim[i,\,s](\Derb(\cor_{X_i}),\De_s^{-1} ).
\eneqn
\end{definition}
Hence, an object $\Derb(\cor_X)$ is the data of $\{(F_i,\varphi_s)\}_{i,s}$
($i\in\cI,s\in\Mor(\cI)$) with $F_i\in\Derb(\cor_{X_i})$
and $\varphi_s\cl\opb{\De_s}F_j\isoto F_i$ satisfying the compatibility conditions in def $\ref{def:2tprolim}$.


\subsubsection*{Inverse images}
Let  $f\cl X\to Y$ be a morphism of Ran manifolds.

\begin{definition}\label{def:extoper}
The inverse image functor $\opb{f}\cl\Derb(\cor_Y) \to \Derb(\cor_X)$ is defined as follows.
 Given $G \in \Derb(\cor_Y), G=\{(G_{i},\psi_s)\}_{i,s}$ with $i\in\cI,s\in\Hom[\cI](j,i)$,
we set $\opb{f}G=\{(F_{i},\varphi_s)\}_{i,s}$ where $F_i=\opb{f_i}G_i$ and
$\varphi_s\cl \opb{\De_s}F_j\simeq F_i$ is given by the isomorphisms
\eqn
&&\opb{\De_s}F_j=\opb{\De_s}\opb{f_j}G_j\simeq\opb{f_i}\opb{\De_s}G_j\simeq\opb{f_i}G_i=F_i.
\eneqn
\end{definition}


%
%

\section{$\shd$-modules on complex Ran-manifolds}

The $\infty$-category of $\shd$-modules on Ran spaces has been introduced and studied  in~\cite{FG11}.
We propose here a more elementary approach.\\ First we shall adapt Definition~\ref{def:shvonR}
in the $\shd$-modules setting.
Let $X$ be a Ran-manifold and $\cI$ a small category.
Define the functor $\sts\cl\cI^{op}\to \Cat$ by
\banum
\item to $i\in \cI$, $\sts(i)=\Derb(\shd_{X_{i}});$
\item for any morphism $s\cl i_{1}\to i_{2}$, $\restr_s:\Derb_{\coh}(\shd_{X_{i_1}})\to\Derb_{\coh}(\shd_{X_{i_2}})$ is the functor $\dopb{\De}_s;$
\item for $s\cl i_{1}\to i_{2},\;t\cl i_{2}\to i_{3}$, the isomorphism of functors $c_{s,t}$ is the obvious one
$\dopb{\De}_t\circ\dopb{\De}_s\simeq\dopb{\De}_{t\circ s}$.
\eanum
\begin{definition}We set
$$\Derb(\shd_{X})=\tprolim[i,\,s](\Derb(\shd_{X_i}),\dopb{\De}_s).$$
\end{definition}
\begin{definition}Denote by $\Derb_{\coh\;\mathrm{nc}}(\shd_{X_j})$ the full triangulated subcategory of $\Derb_{\coh}(\shd_{X_j})$ whose objects $\shm_{j}$ are such that for every $ s\in \Hom[\cI](j,i)$ the corresponding morphism $\De_s\cl X_i\to X_j$ is non-characteristic with respect to  $\shm_{j}$.
\end{definition}
\begin{definition}We set
$$\Derb_{\coh\;\mathrm{nc}}(\shd_{X})=\tprolim[i,\,s](\Derb_{\coh\;\mathrm{nc}}(\shd_{X_i}),\dopb{\De}_s).$$
\end{definition}

Let $X$ and $Y$ be two complex Ran-manifolds and $f\cl X \to Y$.
We shall mimic the previous definitions
\begin{definition}Denote with $\Derb_{\coh\;f-\mathrm{nc}}(\shd_{Y_j})$ the full triangulated subcategory of $\Derb_{\coh}(\shd_{Y_j})$ whose objects $\shn_{j}$ are such that for every $ s\in \Hom[\cI](j,i)$ the corresponding morphism $f_{s}=f_{j}\circ \De^{(X)}_s\cl X_i\to Y_j$ is non-characteristic with respect to $\shn_{j}$.
\end{definition}

\begin{definition}\label{def:shvonC}
 We set $\Derb_{\coh\;f-\mathrm{nc}}(\shd_{Y})$ as the $2$-limit category
$$\tprolim[i,\,s](\Derb_{\coh\;f-\mathrm{nc}}(\shd_{Y_i}),\dopb{\De}_s).$$
\end{definition}
Of course the two definitions coincide in the case $f=id$.
\begin{example}\label{exa:shoXD}
Let $X$ be a complex Ran-manifold .
Since $\dopb{\De}_s\sho_{X_j}\simeq \sho_{X_i}$, the family
$\{\sho_{X_i}\}_i$ defines an object in $\Derb_{\coh}(\shd_X)$.
\end{example}
\begin{lemma} Any object $\shn \in \Derb_{\coh\;f-\mathrm{nc}}(\shd_{Y})$ belongs to the subcategory of $\Derb_{\coh}(\shd_{Y})$ consisting of objects s.t. for every $s\cl j\to i$,
$\De_s$ is non-characteristic for $\shn_{j}$ in a neighborhood of $f_{i}(X_i)$.
\end{lemma}
\begin{proof}{It is a consequence of Lemma \ref{transversal lemma}.}
\end{proof}
\subsubsection*{The inverse image functor $\dopb{f}$}
Let $f\cl X\to Y$ be a morphism of Ran-manifolds.
We will need the following result (which holds in view of Lemma \ref{transversal lemma}).
\begin{lemma}\label{preliminary lemma} For all $s \cl j \to i$ the corresponding
map $\De_s \cl X_{i}\to X_{j}$ is non characteristic for the corresponding $\dopb{f}_{j}\shn_{j}$.
\end{lemma}

\begin{proposition}\label{le:mdfncthick}
The map ${f}$ induces a functor $\dopb{f}\cl \Derb_{\coh\;f-\mathrm{nc}}(\shd_Y)\to \Derb_{\coh\;\mathrm{nc}}(\shd_X)$
\end{proposition}
\begin{proof}
We set $\dopb{f}:=\{\dopb{f}_{i}\}_{i\in \cI}$. This inverse image functor is well-defined due to Lemma \ref{preliminary lemma}.
In view of Lemma \ref{lemma1} it is enough to prove that $\{\dopb{f}_{i}\}_{i\in\cI}$ is a functor of prestack.
In particular this means that each functor $\dopb{f}_i$ commutes with the restriction morphisms. By Definition \ref{mmorphism} of morphism of Ran manifolds we have $ {\De}_s \circ f_{i}=f_{j}\circ  {\De}_s$ and applying the inverse image functor we obtain the desired property $\dopb{f}_{i}\circ \dopb{\De}_s \simeq\dopb{\De}_s\circ \dopb{f}_{j}$.
\end{proof}

\subsubsection*{The functor $\rhom[\shd_{Y}](-,\sho_{Y})$}
\begin{lemma}
The functor $\rhom[\shd_{Y}](-,\sho_{Y}):=\{\rhom[\shd_{Y^i}](-,\sho_{Y^i})\}_{i\in I},$ $$\rhom[\shd_{Y}](-,\sho_{Y})\cl \Derb_{\coh}(\shd_Y)\to \Derb_{\coh}(\cor_X)$$
is well defined.
\end{lemma}
\begin{proof}{Consider an object $\{\shn_{i},\varphi_s\}_{i,s}\in \Derb_{\coh}(\shd_Y)$; by the CKK theorem (Theorem \ref{th:CKK}), there is the isomorphism
$\De^{-1}_s\rhom[\shd_{Y_i}](\shn_i,\sho_{Y_i})\simeq \rhom[\shd_{Y_j}](\dopb{\De}_s\shn_i,\sho_{Y_j})$.
By the composition of the last isomorphism with $\varphi_s:\xymatrix{\dopb{\De}_s\shn_i\ar[r]^{\sim}&\shn_j\\}$ we obtain the isomorphism
$\xymatrix{\De^{-1}_s\rhom[\shd_{Y^i}](\shn_i,\sho_{Y_i})\ar[r]^{\sim} & \rhom[\shd_{Y_j}](\shn_j,\sho_{Y_j})}$. Lemma \ref{lemma1} completes the proof.
}
\end{proof}
\subsubsection*{The Cauchy-Kowaleski-Kashiwara theorem for Ran spaces}
\begin{theorem}
There is an isomorphism of functors of prestacks $$\xymatrix{\theta_{\infty}:f^{-1} \circ\rhom[\shd_{Y}](-,\sho_{Y})\ar[r]^{\sim}& \rhom[\shd_{Y}](-,\sho_{Y}) \circ \dopb{f}}.$$
In other words for any $\shn\in \Derb_{\coh\;f-\mathrm{nc}}(\shd_Y)$ we have the isomorphism $$\xymatrix{f^{-1}\rhom[\shd_{Y}](\shn,\sho_{Y})\ar[r]^{\sim}&\rhom[\shd_{X}](\dopb{f}\shn,\sho_{X})\\} $$
functorial in $\shn\in \Derb_{\coh\;f-\mathrm{nc}}(\shd_Y)$.
\end{theorem}
The theorem is visualized by the following quasi commutative diagram:
\eq
\xymatrix{\Derb_{\coh\;f-\mathrm{nc}}(\shd_Y)\ar[d]_{\rhom[\shd_{Y}](-,\sho_{Y})}\ar[r]^-{\dopb{f}}& \Derb_{\coh\;\mathrm{nc}}(\shd_X)\ar[d]^{\rhom[\shd_{X}](-,\sho_{X})}\\
\Derb(\cor_Y)\ar[r]^{f^{-1}} &\Derb(\cor_X). \\}
\eneq

\begin{proof} {In view of Lemmas \ref{lemma2}, \ref{lemma2}, \ref{lemma4}, it is enough to observe that each morphism $$\xymatrix{\theta_{i}:f^{-1}_{i} \circ\rhom[\shd_{Y_i}](-,\sho_{Y_i})\ar[r]^{\sim}& \rhom[\shd_{Y_i}](-,\sho_{Y_i}) \circ \dopb{f}_i}$$
is an isomorphism due to Theorem \ref{th:CKK}.}
\end{proof}

\subsubsection*{}

\begin{example}
 Let $\cI$ be the category of finite non-empty sets and
surjective maps. Let $X$ be a complex manifold.
The Ran space $Ran\;X$ in \cite{BD04} is defined as follows:
to $I\in\cI$ one associates the product manifold $X^I$ and to a surjection
$s\cl J\to I$ is associated the diagonal embedding
\eq
&&\delta_s\cl X^I\into X^J
\eneq
which maps $\{x_i\}_{i\in I}\in X^I$ to $\{x_j\}_{j\in J}\in X^J$ with
$x_j=x_i$ if $s(j)=i$.
\end{example}

\subsection*{Appendix}
There are two other natural categories, $\sts^+(\cI)$ and $\sts^-(\cI)$.

\begin{definition}\label{def:2tprolimpm}
\banum
\item
An object $F$ of $\sts^+(\cI)$ \lp resp.\ $\sts^-(\cI)$\rp\,
 is a family $\{(F_i, \phi_s)\}_{i,s}$
($i\in\cI$, $s\in\Fl(\cI)$) where
\bnum
\item
for any $i\in\cI$, $F_i$ is an object of $\sts(i)$,
\item
for any morphism $s\colon i_1\to i_2$ in $\cI$,
$\phi_s \cl F_{i_1} \to \restr_s (F_{i_2})$
\lp resp.\ $\phi_s \cl\restr_s (F_{i_2}) \to F_{i_1}$\rp\,
is a morphism such that
\begin{itemize}
\item
for all $i \in \cI$, $\phi_{\id_i} = \id_{F_i}$,
\item
for any sequence $i_1\to[s]i_2\to[t]i_3$ of morphisms in $\cI$,
the following diagram commutes
\eq\label{cond:cos_S(A)}
&& \vcenter{
\xymatrix{
{\restr_s \restr_t}(F_{i_3}) \ar@{-}[d]_-{c_{s,t}}^-\wr
           &&{\restr_s} (F_{i_2}) \ar[ll]_-{\restr_s(\phi_t)}\\
{\restr_{t\circ s}} (F_{i_3})
       &&F_{i_1}\ar[u]_-{\phi_{s}}\ar[ll]^-{\phi_{t\circ s}}
}}
\eneq
\lp resp.\
the following diagram commutes
\eq\label{cond:cos_S(A)b}
&& \vcenter{ \xymatrix{
\restr_s \restr_t (F_{i_3})
\ar@{-}[d]_-{c_{s,t}}^-\wr\ar[rr]^-{\restr_s(\phi_t)}
        &&\restr_s (F_{i_2})   \ar[d]^-{\phi_{s}}\\
\restr_{t\circ s} (F_{i_3}) \ar[rr]_-{\phi_{t\circ s}}
       &&F_{i_1}.)
}}
\eneq
\end{itemize}
\enum
\item
A morphism
$f\cl\{(F_i,\phi_s)\}_{i,s}\to\{(F'_i,\phi'_s)\}_{i,s}$
in $\sts^+(\cI)$ \lp resp.\ $\sts^-(\cI)$\rp\, is
a family of morphisms
$f_i\cl F_i\to F'_i$ such that
for any $s\cl i_1\to i_2$, the diagram below commutes:
\eq\label{diag:morsts+}
\vcenter{ \xymatrix{
\restr_s (F_{i_2})  \ar[rr]^-{\restr_s (f_{i_2}) }
             &&\restr_{s}(F'_{i_2})  \\
F_{i_1}\ar[rr]^-{f_{i_1}}  \ar[u]^-{\phi_s}
&&F'_{i_1} \ar[u]_-{\phi'_s} .  }
}
\eneq
\lp resp.\ the diagram below commutes:
\eq\label{diag:morsts-} \vcenter{ \xymatrix{
\restr_s (F_{i_2})  \ar[rr]^-{\restr_s (f_{i_2})}\ar[d]_-{\phi_s}
             &&\restr_{s}(F'_{i_2}) \ar[d]^-{\phi'_s}  \\
F_{i_1}\ar[rr]^-{f_{i_1}}
&&F'_{i_1} . ) }
}
\eneq
\item
We consider $\sts(\cI)$ as the full
subcategory of $\sts^+(\cI)$ or $\sts^-(\cI)$
consisting of
objects $\{(F_i, \phi_s)\}_{i\in\cI,s\in\Fl(\cI)}$ such
that for all $s\in\Fl(\cI)$, the
 morphisms $\phi_s$ are isomorphisms and
we denote by $\iota^+_\cI\cl \sts(\cI) \to \sts^+(\cI)$ and
$\iota^-_\cI\cl \sts(\cI) \to \sts^-(\cI)$ the natural
 faithful functors.
\eanum
\end{definition}

\providecommand{\bysame}{\leavevmode\hbox to3em{\hrulefill}\thinspace}

Giuseppe Bonavolont\`a,\\
Campus Kirchberg, Mathematics Research Unit\\ 6, rue R. Coudenhove-Kalergi, L-1359 Luxembourg
City
\end{document}